\documentclass[12pt]{amsart}
\usepackage{amssymb,amsmath,amsthm,latexsym}
\usepackage{mathrsfs}
\usepackage{a4wide}
\usepackage[usenames,dvipsnames]{color}
\usepackage{euscript}
\usepackage{graphicx}
\usepackage{comment}
\usepackage{mdwlist}
\usepackage{mathtools,dsfont,wasysym}
\usepackage[makeroom,samesize]{cancel}

\usepackage[dvipsnames]{xcolor}

\DeclareFontFamily{U}{mathx}{\hyphenchar\font45}
\DeclareFontShape{U}{mathx}{m}{n}{
      <5> <6> <7> <8> <9> <10>
      <10.95> <12> <14.4> <17.28> <20.74> <24.88>
      mathx10
      }{}

\newcommand{\nn}[1]{{\vert\kern-0.25ex\vert\kern-0.25ex\vert #1 
    \vert\kern-0.25ex\vert\kern-0.25ex\vert}}

\DeclareMathOperator{\supp}{supp}

\newtheorem{theorem}{Theorem}
\newtheorem{lemma}[theorem]{Lemma}

\theoremstyle{remark}
\newtheorem{remark}[theorem]{Remark}

\newtheorem*{remark*}{Remark}
\theoremstyle{definition}
\newtheorem{definition}[theorem]{Definition}

\numberwithin{equation}{section}

\newcounter{maintheorem}

\newtheorem{mainth}[maintheorem]{Theorem}
\newtheorem*{mainthprime*}{Theorem A$^\prime$}
\makeatother

\renewcommand{\leq}{\leqslant}
\renewcommand{\geq}{\geqslant}
\newcounter{smallromans}

{\end{list}}

\newcounter{smallromansdash}

{\end{list}}

\newcounter{bigromans} 
  {\end{list}}

\begin{document}

\baselineskip=17pt

\title{Isometries of combinatorial Tsirelson spaces}
\author{Natalia Maślany}
\address{ Jagiellonian University, Doctoral School of Exact and Natural Sciences, Faculty of Mathematics and Computer Science, Institute of Mathematics, {\L}ojasiewicza 6, 30-348 Krak\'ow, Poland and Institute of Mathematics, Czech Academy of Sciences, \v{Z}itn\'{a} 25, 115~67 Prague 1, Czech Republic}
\email{nataliamaslany97@gmail.com}
\thanks{The author was supported by GAČR grant GF20-22230L and received funding in the scope Research Support Module at the Jagiellonian University in Krak\'{o}w. }

\begin{abstract}
    We extend existing results that characterize isometries on the Tsirelson-type spaces $T\big[\frac{1}{n}, \mathcal{S}_1\big]$ ($n\in \mathbb{N}, n\geq 2$) to the class $T[\theta, \mathcal{S}_{\alpha}]$ \big($\theta \in \big(0, \frac{1}{2}\big]$, $1\leqslant \alpha < \omega_1$\big), where $\mathcal{S}_{\alpha}$ denote the Schreier families of order $\alpha$. 
    We prove that every isometry on $T[\theta, \mathcal{S}_1]$ \big($\theta \in \big(0, \frac{1}{2}\big]$\big) is determined by a permutation of the first $\lceil {\theta}^{-1} \rceil$ elements of the canonical unit basis followed by a possible sign-change of the corresponding coordinates together with a sign-change of the remaining coordinates. Moreover, we show that for the spaces $T[\theta, \mathcal{S}_{\alpha}]$ \big($\theta \in \big(0, \frac{1}{2}\big]$, $2\leqslant \alpha < \omega_1$\big) the isometries exhibit a more rigid character, namely, they are all implemented by a sign-change operation of the vector coordinates. 
\end{abstract}


\maketitle
\section{Introduction and the main result}

The well-known Tsirelson space $T$ (in the setting of Figiel and Johnson \cite{FigielJohnson1974}, i.e., the dual of the space constructed by Tsirelson \cite{tsirelson1974impossible}, the first example of a space containing no isomorphic copies of $c_0$ or $\ell_p$ for $1\leq p < \infty$) may be regarded as special instance of a space from a~double-parameter family of Banach spaces $T[\theta, \mathcal{S}_{\alpha}]$ \big($\theta \in \big(0, \frac{1}{2}\big]$, $1\leqslant \alpha < \omega_1$\big), where $\alpha$ is a~countable ordinal and  $S_\alpha$ is the Schreier family of order $\alpha$. We emphasize that in this paper the Schreier families are defined to satisfy some natural additional conditions (see the discussion after Definition \ref{d3} in Preliminaries). For brevity, we call members of this family \emph{combinatorial Tsirelson spaces}, which appears to be in line with the terminology used, e.g., in \cite{BrechFerencziTcaciuc2020}. (These are, of course, special cases of the so-called mixed Tsirelson spaces whose definition allows the parameter $\theta$ to vary, but by employing this name we want to emphasize the underlying family of sets rather than the numeric parameter.)\smallskip

The aim of this paper is to delineate the structure of isometries on combinatorial Tsirelson spaces. We refer to the recent excellent survey \cite{antunes2022surjective} for further results and references concerning the problem of characerization of isometries on Banach (sequence) spaces.\smallskip

In \cite[Theorem 4.1]{antunes2022surjective} the authors provide a characterization of  (linear) isometries of the spaces $T\big[\frac{1}{n}, \mathcal{S}_1\big]$ for $n\in \mathbb{N}$, $n\geq 2$, which we take as a departure point for our considerations and extend it to the whole scale of spaces $T\big[\theta, \mathcal{S}_1\big]$ \big($\theta \in \big(0,\frac{1}{2} \big]\big).$ Let us then record the first main result. (In this paper, all considered Banach spaces are real; by an \emph{isometry} we understand a~linear isometry.)
\begin{mainth}\label{Th:A}
    Let $\theta \in \big(0,\frac{1}{2} \big]$. If $U\colon T\big[\theta, \mathcal{S}_1\big]\to T\big[\theta, \mathcal{S}_1\big]$ is an isometry, then
    \[
        U e_i = \left\{\begin{array}{ll} \varepsilon_i e_{\pi(i)}, &  1\leq i\leq \lceil \theta^{-1}  \rceil\\
        \varepsilon_i e_i, & i > \lceil \theta^{-1}  \rceil
        \end{array} \right. \quad(i\in \mathbb N)
    \]
    for some $\{-1,1\}$-valued sequence $(\varepsilon_i)_{i=1}^\infty$ and a permutation $\pi$ of $\big\{1,2,\ldots, \lceil \theta^{-1}  \rceil\big\}.$
\end{mainth}
(Here $(e_i)_{i=1}^\infty$ is the standard unit vector basis of $T\big[\theta, \mathcal{S}_1\big]$ and $\lceil \theta^{-1} \rceil$ is the ceil of $\theta^{-1}$, i.e., the least integer that $\theta^{-1}$ does not exceed.)  

Then we answer the question from \cite{antunes2022surjective}.
Indeed, we characterize the linear isometries of the spaces  $T[\theta, \mathcal{S}_{\alpha}]$ ($\theta \in \big(0, \frac{1}{2}\big]$, $2\leqslant \alpha < \omega_1$) by proving the following second main result:
\begin{mainth}\label{Th:B}
    Let $\theta \in \big(0,\frac{1}{2} \big]$ and let $\alpha\geq 2$ be a countable ordinal. Then an operator $U\colon T[\theta,\mathcal{S}_{\alpha}] \to T[\theta,\mathcal{S}_{\alpha}]$ is an isometry if and only if $Ue_i = \varepsilon_i e_i$ for $i \in \mathbb{N}$ and some $\{-1,1\}$-valued sequence $(\varepsilon_i)_{i=1}^\infty.$
\end{mainth}

Let us record the following observation that we draw directly from the proofs of Theorems \ref{Th:A}--\ref{Th:B}, which may be of independent interest.
\begin{remark}
     Every isometry on $T[\theta, \mathcal{S}_{\alpha}]$ \big($\theta \in \big(0, \frac{1}{2}\big]$, $1\leqslant \alpha < \omega_1$\big) is surjective.
\end{remark}
\subsection*{Acknowledgement.} We would like to thank Kevin Beanland for helpful conversations and feedback concerning the present work.

\section{Preliminaries}

\subsection{Combinatorial spaces}
We will denote by $(e_i)_{i=1}^\infty$ the standard unit vector basis of $c_{00}$ and by $[\mathbb{N}]^{<\omega}$ the family of finite subsets of $\mathbb{N}.$ For the sets $F_1, F_2 \in [\mathbb{N}]^{<\omega}$ we use the following notation: $F_1<F_2$, whenever $\max F_1 < \min F_2$ and in such case we say that these sets are \textit{consecutive}. Moreover, for $n\in \mathbb{N},$ we write $F_1 < n$ instead of $F_1 < \{n\}.$

\begin{definition}\label{sp}
A family $\mathcal{F}\subset [\mathbb{N}]^{<\omega}$ is {\em regular}, whenever it is simultaneously
\begin{itemize}
		\item \textit{hereditary} \big($F \in \mathcal{F}$ and $G \subset F \implies G \in \mathcal{F}$\big);
		\item \textit{spreading} \big($\{l_1,l_2,\dots,l_n\} \in \mathcal{F}$ and $l_i \leqslant k_i \implies \{k_1,k_2,\dots,k_n\} \in \mathcal{F}$\big);
		\item \textit{compact} as a subset of the Cantor set $\{0,1\}^{\mathbb{N}}$ via the natural identification of $F \in \mathcal{F}$ with 
        \[
            \chi_F=\sum_{i \in F} e_i \in \{0,1\}^{\mathbb{N}}.
        \]
\end{itemize}
\end{definition}
If $\mathcal{F}$ is a regular family, we say that $F\in \mathcal{F}$ is \textit{maximal}, whenever there is no $n\in \mathbb{N}$ with $\max F <n$ so that $F\cup \{n\}\in \mathcal{F}$. The simplest examples of regular families include
\[
    \mathcal{A}_n := \big\{F\in[\mathbb{N}]^{<\omega} \colon |F|\leqslant n\big\} \quad(n\in \mathbb N)
\]
i.e., for a given $n\in \mathbb N$, the family comprising subsets of $\mathbb N$ of cardinality at most $n$. We employ these families to define the family of Schreier sets in the following way.
\begin{definition}\label{d3}
 Given a countable ordinal $\alpha$, we define inductively the Schreier family of order $\alpha$  as follows:
 \begin{itemize}
     \item $\mathcal{S}_{0} := \mathcal A_1;$
     \item if $\alpha$ is a successor ordinal, i.e., $\alpha=\beta +1$ for some $\beta < \omega_1$, then 
        \[
            \mathcal{S}_{\alpha}:=\Bigg\{ \bigcup_{i=1}^{d} S^{i}_{\beta}\colon d \leq S^{1}_{\beta} < S^{2}_{\beta} < \cdots < S^{d}_{\beta},\,\, \big\{ S^{i}_{\beta} \big\}_{i=1}^d \subset \mathcal{S}_{\beta} \text{ and } d \in \mathbb{N} \Bigg\} \cup\big\{\emptyset\big\};
        \]
     \item if $\alpha$ is a non-zero limit ordinal and $\left(\alpha_{n}\right)_{n=1}^\infty$ is a fixed strictly increasing 
     sequence of successor ordinals converging to $\alpha$ with 
     $\mathcal S_{\beta_n} \subset \mathcal S_{\beta_{n+1}}$
     for all $n \in \mathbb{N}$, where $\alpha_n = \beta_n+1$ for all $n \in \mathbb{N}$, we set 
        \begin{equation*}
        \mathcal{S}_{\alpha} :=\big\{S_{\alpha_n} \in[\mathbb{N}]^{<\omega}\colon S_{\alpha_n} \in    \mathcal{S}_{\alpha_{n}},\,\, n \leq \min S_{\alpha_n} \text{ for some n} \big\} \cup\big\{\emptyset\big\}. 
        \end{equation*}
 \end{itemize}
\end{definition}

We emphasize that in the case where $\alpha$ is a limit ordinal, we require the sequence $(\alpha_n)_{n=1}^\infty$ cofinal in $\alpha$ to comprise successor ordinals as needed in the proof of Theorem \ref{Th:B}. We may (and do) also assume that $S_{\alpha_n} \subset S_{\alpha_{n+1}}$ for all $n \in \mathbb{N}$, which we will also utilize in the proof of Theorem \ref{Th:B}. Indeed, repeating the proof of \cite[Proposition 3.2.]{causey2017concerning} in the case of Schreier families $\big\{ S_{\xi} \big\}_{\xi < \omega_1}$ which are multiplicative in the sense of \cite{causey2017concerning} we obtain the required result also in the case of Schreier families.
Elements of $\mathcal S_{\alpha}$ are called $\mathcal S_{\alpha}$\emph{-sets}.

Note that the Schreier families $\left\{\mathcal{S}_{\alpha}\right\}_{\alpha<\omega_{1}}$ do depend on the choice of the sequences $\left\{\alpha_{n}\right\}_{n=1}^{\infty}$ converging to each limit ordinal $\alpha$. It is a well-known fact \big(\cite{causey2017concerning}[Proposition 3.2] or \cite{todorcevic2005ramsey}\big) that they are always regular families. 

\subsection{Combinatorial Tsirelson spaces}
For a regular family $\mathcal{F}$ and $\theta\in \big(0,\frac{1}{2}\big]$, we define the Banach space $T[\theta,\mathcal{F}]$ that we shall later specialize to a combinatorial Tsirelson space $T[\theta, \mathcal{S}_{\alpha}]$ for some countable ordinal $\alpha$.\smallskip

For a vector $x=(a_1, a_2, \ldots, a_n)\in c_{00}$ and a finite set $E\subset \mathbb N$, we use the same symbol $Ex$ to denote the projection of $x$ onto the space $[e_i \colon i \in E],$ given by
\begin{equation}
    E\Bigg(\sum_{i=1}^{n} a_i e_i\Bigg) = \sum_{i\in E} a_i e_i.
\end{equation}

If the cardinality of set $E$ is equal $k,$ then we say that $E$ is a \emph{$k$-element projection}.

We denote by $\|\cdot\|_0$ the supremum norm on $c_{00}.$ Suppose that for some $n \in \mathbb{N}$ the norm $\|\cdot\|_n$ has been defined. Let
\begin{equation*}\label{nnorm} 
 \|x\|_{n+1}= \max \big\{ \|x\|_n, \|x\|_{T_n} \big\}\quad (n\in \mathbb N),
\end{equation*}
where
\begin{equation*}
 \|x\|_{T_n}= \sup \bigg\{ \theta \sum_{i=1}^d \big\|E_ix\big\|_n\colon E_1 <\cdots < E_d,\, d \in \mathbb{N},\, \{E_i\}_{i=1}^d \subset [\mathbb{N}]^{<\omega}, \, \{\min E_{i}\}_{i=1}^d \in \mathcal{F}\bigg\}.  
\end{equation*}
We define the norm $\|x\|_{\theta,\mathcal{F}} := \sup_{n \in \mathbb{N}}\|x\|_n$ and denote by $T[\theta,\mathcal{F}]$ the completion of $c_{00}$ with respect to it. 

A proof by induction shows that this norm is majorized by the $\ell_1$-norm and that satisfies the following implicit formula for $x \in T[\theta,\mathcal{F}]$: 
\begin{equation}\label{im}
    \|x\|_{\theta,\mathcal{F}}= \max \big\{\|x\|_{\infty} , \|x\|_T \big\},
\end{equation}
where
\begin{equation*}
    \|x\|_T= \sup\bigg\{\theta \sum_{i=1}^d \big\|E_ix\big\|_{\theta,\mathcal{F}} : E_1 < \cdots < E_d,\, d \in \mathbb{N},\, \{E_i\}_{i=1}^d \subset [\mathbb{N}]^{<\omega}, \,\{\min E_i\}_{i=1}^d \in \mathcal{F}\bigg\}.
\end{equation*}

It follows easily from the definition that the unit vectors $(e_i)_{i=1}^\infty$ form an 1-unconditional basis of the space  $T[\theta,\mathcal{S}_{\alpha}]$ for a countable ordinal $\alpha$.

For $x_1,x_2 \in c_{00}$, we write $x_1<x_2$ whenever $\supp\, x_1 < \supp\, x_2$ and for $n\in\mathbb{N}$ we simplify the notation of $\supp\, x_1<n$ to $x_1<n.$

In this paper, we will use the following convention: we say that the norm of an element $x \in T[\theta,\mathcal{F}]$ is given by sets $E_1<E_2<\cdots < E_{d}$ for some $d\in \mathbb{N}$ (with $\{\min E_i\}_{i=1}^d \in \mathcal{F}$) precisely when
\begin{equation*}
    \|x\|_{\theta,\mathcal{F}}= \theta \cdot \sum_{i=1}^{d}\big\|E_i x\big\|_{\theta,\mathcal{F}}.
\end{equation*}

It follows easily from the definition that if $(x_i)_{i=1}^d$ is a block sequence in $T[\theta,\mathcal{F}]$ (i.e., $x_1<x_2<\ldots<x_d$) with $\{\min\text{supp}\,x_i\}_{i=1}^d\in \mathcal{F}$ we have  
\begin{equation}
\label{w1}
 \bigg\|\sum_{i=1}^n x_i\bigg\|_{\theta,\mathcal{F}} \geqslant \theta\sum_{i=1}^n\|x_i\|_{\theta,\mathcal{F}}. 
% \label{genlower}
\end{equation}
It follows that in the case of the space $T[\theta,\mathcal{S}_1]$, \eqref{w1} yields that if $d\leqslant x_1<\cdots <x_d$, then 
\begin{equation}\label{w2}
     \bigg\|\sum_{i=1}^d x_i\bigg\|_{\theta,\mathcal{S}_1} \geqslant \theta\sum_{i=1}^d\|x_i\|_{\theta,\mathcal{S}_1}   
\end{equation}
For brevity, we write $\|\cdot \|$ instead of $\|\cdot \|_{\theta,\mathcal{S}_{\alpha}},$ where $\theta \in \big(0, \frac{1}{2}\big]$, $1\leqslant \alpha < \omega_1.$ Let us record the following lemma that we shall later use extensively.

\begin{lemma}\label{obs}
    Let $\theta \in \big(0, \frac{1}{2}\big]$ and $1\leqslant \alpha < \omega_1$. Suppose that $x \in T[\theta, \mathcal{S}_{\alpha}]$ is a vector whose coordinates are either $0$ or $1$. Fix $k\geq 2.$ Let $E$ be a $k$-element set. Then the norm given by the set $E$ is not greater than the norm given by the $k-1$ many singleton projections. (For simplicity we assume we only project to non-zero coordinates and all of these projections are admissible.)
\end{lemma}
\begin{proof}
    Fix $k\geq 2.$ The norm given by $k-1$ many $1$-element projections is $\theta (k-1)$.
    Let $E$ be a $k$-element set and let $E_1<E_2<\dots<E_d$. Then 
    \begin{equation*}
    \begin{split}
    \theta \cdot \| Ex \| &= \theta \cdot \max \bigg\{ 1,\, \theta \cdot \sum_{i=1}^d \|E_i x\|  \bigg\} \\
    &\leq \theta \cdot \max \bigg\{ 1,\, \theta \cdot \sum_{i=1}^d \|E_i x\|_{\ell_1}  \bigg\} \\
    &= \theta \cdot \max \bigg\{ 1,\, \theta \cdot \sum_{i=1}^d |E_i| \bigg\} \\
    &\leq \theta \cdot \max \{1, \, \theta k\}.
    \end{split}
    \end{equation*}
    Since $\theta \in \big(0, \frac{1}{2}\big]$ and $k \geq 2$, so $\theta (k-1) \geq \theta^2 k.$
\end{proof}

\begin{lemma}\label{ogl}
    Let $\theta \in \big(0, \frac{1}{2}\big]$ and $1\leqslant \alpha < \omega_1$. Suppose that $x \in T[\theta, \mathcal{S}_{\alpha}]$ is given by the formula
    \begin{equation*}
        x = e_i + \sum_{j\in A} e_j,
    \end{equation*}
    where $A=\big\{ j_1, j_2, \ldots, j_{|A|}\big\}$ is an $\mathcal S_{\alpha}$-set with $A>i$ and $|A|\geq \lceil \theta^{-1} \rceil$. If $A\cup \{i\}$ is not an~$\mathcal S_{\alpha}$-set, then $\|x\|= \theta \cdot |A|.$
\end{lemma}
\begin{proof}
 Since $A$ is an $\mathcal S_{\alpha}$-set, we have 
\begin{equation*}
    \|x\| \geq \theta \cdot |A| \geq 1 = \|x\|_{\infty}.
\end{equation*}
Assume that there are sets $E_1<E_2<\dots<E_d$ which give norm of $x$ greater than $\theta \cdot |A|$. If $\min E_1 > i,$ then by Lemma \ref{obs} we arrive at a contradiction. If $\min E_1 \leq i,$ then we have at least one $k$-element set ($k\geq 2$), because by the hypothesis, $A\cup \{i\} \notin S_{\alpha}.$ This  contradicts Lemma \ref{obs} likewise. 

\end{proof}

\section{Isometries on $T[\theta, \mathcal{S}_1]$ spaces for $\theta \in \big(0, \frac{1}{2}\big]$ }

For the space $T[\theta, \mathcal{S}_1]$ $\big(\theta \in (0,1)\big)$ the norm given by the formula \eqref{im} takes the following form:

\begin{equation}\label{nn}
 \|x\|= \max \big\{\|x\|_{\infty} , \|x\|_{T}\big\}, 
\end{equation}
 where
 \begin{equation}\label{nt}
 \|x\|_{T} = \sup \bigg\{\theta \sum_{i=1}^d \big\|E_ix\big\| : d \leq E_1 < E_2<\cdots < E_d, \, d \in \mathbb{N},\, \{E_i\}_{i=1}^d \subset [\mathbb{N}]^{<\omega} \bigg\}.
 \end{equation}

We are now ready to prove Theorem~\ref{Th:A}; the proof emulates the one of \cite[Theorem 4.1]{antunes2022surjective}.

\begin{proof}
Let $Ue_n := \sum_{i=1}^\infty a^n_i e_i$ ($n \in \mathbb{N}$).
\begin{itemize}
\item[] \emph{Claim 1}. 
For any $n\geq \lceil \theta^{-1} \rceil$ we have $U([e_1,e_2,\ldots,e_n]) \subset [e_1,e_2,\ldots,e_n].$

Let $n\geq \lceil \theta^{-1} \rceil$ and $j \in \big\{1, 2, \dots,n\big\}.$ Define $x:=\sum_{i={n}+1}^\infty a^j_i e_i $ and fix $\varepsilon>0.$

As $(Ue_i)_{i=1}^\infty$ is weakly null we may find indices 
$
n<j_1<j_2<\cdots <j_{n}
$
and vectors 
\[
n+1 \leq x' < y_1 < y_2 <\cdots <y_{n},
\] 
so that $\|x'-x\|< \varepsilon$ and $ \big\| Ue_{j_i}- y_i\big\|<\varepsilon$ for $1\leq i \leq {n}$. 

By Lemma \ref{ogl} and since $U$ is an isometry we have
\begin{equation*}
      \theta n = \Bigg\|e_j + \sum_{i=1}^{n} e_{j_i} \Bigg\| =  \Bigg\|\sum_{i=1}^{n} a^j_i e_i + x + \sum_{i=1}^{n} U e_{j_i}\Bigg\| 
\end{equation*}      
Hence by triangle inequality we obtain
\begin{equation}\label{35}
    \Bigg\|\sum_{i=1}^{n} a^j_i e_i + x' + \sum_{i=1}^{n} y_i\Bigg\|\leq \theta n+({n}+1)\varepsilon  
\end{equation}
On the other hand
\begin{equation}\label{36}
     \Bigg\|\sum_{i=1}^{n} a^j_i e_i + x' + \sum_{i=1}^{n} y_i\Bigg\|  \geq  \theta \Bigg(\|x'\| + \sum_{i=1}^{n} \|y_i\|\Bigg) >  \theta \Big(\|x'\| + n(1-\varepsilon)\Big). 
\end{equation}
The first inequality follows from \eqref{w2} and second by the fact that $\|Ue_i\|=1$ for any $i\in \mathbb{N}.$ 
Therefore by \eqref{35} and \eqref{36} we obtain
\begin{equation*}
    \|x\|\leq \|x'-x\|+\|x'\| < \varepsilon + \Big(\big({n}+1\big)\theta^{-1} + {n} \Big)\varepsilon. 
\end{equation*}
Since $\varepsilon$ was arbitrary, we get $\|x\|=0.$
Consequently,
\[
    U\big[e_1, e_2, \ldots, e_{n}\big] \subset \big[e_1, e_2, \ldots, e_{n}\big].
\]
 \item[] \emph{Claim 2}. There exists a permutation $\pi$ of $\big\{1, 2, \ldots,{\lceil \theta^{-1} \rceil}\big\}$ such that $Ue_n=\pm e_{\pi(n)}$ for $n \in \big\{1, 2, \ldots ,{\lceil \theta^{-1} \rceil}\big\}.$
 
 First we will show that the norm on $\big[e_1, e_2, \ldots, e_{\lceil \theta^{-1} \rceil}\big]$ is the supremum norm. Indeed, suppose that the norm of some
 \[
    x=\sum_{i=1}^{\lceil \theta^{-1} \rceil} a_i e_i
 \]
 is given by certain sets $d \leq E_1<E_2<\cdots < E_{d}$ for some $d\in \mathbb{N}$ in the sense that
\begin{equation}\label{bb}
    \|x\|= \theta \cdot \sum_{i=1}^{d}\big\|E_i x\big\|.
\end{equation}

Suppose that $\min E_1 < \lceil {\theta^{-1}} \rceil.$ Then
\begin{equation*}
   d \leq \min E_1 \leq \lceil {\theta^{-1}} \rceil - 1 < {\theta^{-1}},
\end{equation*}
so
\begin{equation*}
    \theta \cdot \sum_{i=1}^{d}\big\|E_ix\big\|\leq \theta \cdot d \cdot \|x\| < \|x\|.
\end{equation*}
Hence \eqref{bb} cannot hold; a contradiction.
Suppose that $\min E_1 \geq \lceil \theta^{-1} \rceil.$ Then the only non-zero coordinate of $(E_1 \cup E_2 \cup \ldots \cup E_d)x$ is $a_{\lceil \theta^{-1} \rceil}$, so 
\[
    \|x\| \leq \theta \cdot |a_{\lceil \theta^{-1} \rceil}| \leq \theta \cdot \|x\|_{\infty} < \|x\|_{\infty}.
\]
This contradiction ends the proof that $\|x\|=\|x\|_{\infty}.$ This means that for each  $n \in \big\{1, 2, \ldots ,{\lceil \theta^{-1} \rceil}\big\}$ there is at least one index $\pi(n)$ so that $\big|a^n_{\pi(n)}\big|=1.$  
By the very definition of the norm and since $U$ is an isometry we have 
\begin{equation*}
    \qquad \quad \, 1=\max \{\, 1,\, 2 \cdot \theta \,\} = \|e_n\pm e_i\|= \big\|U(e_n\pm e_i)\big\| \geq \big\|U(e_n\pm e_i)\big\|_{\infty} \geq \big|a_{\pi(n)}^n \pm a_{\pi(n)}^i\big|    
\end{equation*}
 for any $i\neq n$ in $\big\{1,\ldots, {\lceil \theta^{-1} \rceil}\big\}.$ Therefore $\big|a^i_{\pi(n)}\big|=0$ for any  $i\neq n,$ so $\pi$ is the desired permutation. 
\smallskip
 \item[] \emph{Claim 3}. $Ue_n = \pm e_n$ for $n> \lceil \theta^{-1} \rceil.$

 Let $n=\lceil \theta^{-1} \rceil + 1$. Then, by Claim 1, $Ue_{n}= a^{n}_1 e_1 + a^{n}_2 e_2 +  \dots + a^{n}_{{n}}e_{n}$ and, by Claim 2, for $1 \leq j< n$ we have 
\begin{equation*}
    1=\max \{\, 1,\, 2 \cdot \theta \,\} = \|e_j\pm e_{n}\|= \big\|e_{\pi(j)}\pm Ue_{n}\big\| \geq \big|1\pm a^{n}_{\pi(j)}\big|.
\end{equation*}
Consequently, $a^{n}_{\pi(j)}=0$ for all such $j$, so $Ue_{n}=\pm e_{n}.$

 We now proceed to the inductive step. Fix $n\in \mathbb{N},$ $n> \lceil \theta^{-1} \rceil$ and assume that for $k\in \mathbb{N}$ with $\lceil \theta^{-1} \rceil < k < n$ one has $Ue_k=\pm e_k$. Then by Claim 1 we have $Ue_{n}= a^{n}_1 e_1 + a^{n}_2 e_2 +  \dots + a^{n}_{{n}}e_{n}.$ 
By Claim 2, for $1 \leq j\leq \lceil \theta^{-1} \rceil$ we have 
\begin{equation*}
    1=\max \{\, 1,\, 2 \cdot \theta \,\} = \|e_j\pm e_{n}\|= \big\|e_{\pi(j)}\pm Ue_{n}\big\| \geq \big|1\pm a^{n}_{\pi(j)}\big|.
\end{equation*}
Therefore $a^{n}_{\pi(j)}=0$ for all such $j$. 
Similarly, by the inductive hypothesis, we have 
\begin{equation*}
    1=\max \{\, 1,\, 2 \cdot \theta \,\} = \|e_j\pm e_{n}\|= \|e_{j}\pm Ue_{n}\| \geq \big|1\pm a^{n}_j\big|
\end{equation*}
 for $\lceil \theta^{-1} \rceil < j < n,$ so $a^{n}_{j}=0$ for all such $j.$ Hence $Ue_{n}=\pm e_{n}.$ 
This finishes the proof that the isometry has the desired form. 
\end{itemize}\end{proof}

\begin{remark}\label{l1}
    The reverse implication in Theorem~\ref{Th:A} need not hold. Indeed, let ${\theta^{-1}}=2.1.$ Then, of course, $\lceil {\theta^{-1}} \rceil = 3.$ Let us consider the vectors:
\begin{enumerate}
   \item
$
x = \left( 1, 0, 0, 1, 1, 0 , \ldots \right)
$
    \item
    $
y = \left( 0, 0, 1, 1, 1, 0 , \ldots \right)
    $
\end{enumerate}
The only $\mathcal{S}_1$-set with a minimum equal to $1$ is $\{1\}.$ Since $\{4,5\} \in \mathcal{S}_1,$ so from \eqref{nt} we have
\[
\|x\|_{T} \geq \frac{10}{21} \cdot \left(1 + 1\right) = \frac{20}{21}.
\]
Now that vector $x$ has  only $3$ ones, so by Lemma \ref{obs} this inequality is in fact equality.
Since $\{3,4,5\} \in \mathcal{S}_1,$ so again by \eqref{nt} and Lemma \ref{obs} we obtain
\[
\|y\|_{T} = \frac{10}{21} \cdot \left(1 + 1 + 1\right) = \frac{30}{21}.   
\]

Consequently,  by \eqref{nn}, we have
\begin{enumerate}
\item $\|x\|= \max\{\|x\|_{\infty} , \|x\|_{T}\} = \max\{1 , \frac{20}{21}\} = 1,$
\item $\|y\|= \max\{\|y\|_{\infty} , \|y\|_{T}\} = \max\left\{1 , \tfrac{30}{21}\right\} = \frac{30}{21}.$
\end{enumerate}

Let us define an operator $U$ such that $Ux = y$. More formally, 
\[ 
 U e_1 = e_3, \,\,
U e_2 = e_2,\,\, U e_3 = e_1, \,\,
U e_i = e_i \text{ for } i \geq 4.
\]
It is clear that $U$ is not an isometry.
\end{remark}

\section{Isometries on $T[\theta, \mathcal{S}_{\alpha}]$ for $\theta \in \big(0, \frac{1}{2}\big]$ and $2\leqslant \alpha < \omega_1$ }

We are now ready to prove Theorem~\ref{Th:B}.

\begin{proof}
Let $Ue_n := \sum_{i=1}^\infty a^n_i e_i$ ($n \in \mathbb{N}$).

\noindent\emph{Claim 1}. For any ordinal $2\leq\alpha< \omega_1$ and  $n\in \mathbb{N}$ we have $U([e_1,e_2,\ldots,e_n]) \subset [e_1,e_2,\ldots,e_n].$

Fix an ordinal $2\leq \alpha < \omega_1$, $j \in \big\{1, 2, \dots,n\big\},$ where $n \in \mathbb{N}$ and $\varepsilon>0.$ 
Define $x:= \sum_{i=n+1}^\infty a^j_i e_i$ and $t := \max \{n, \lceil \theta^{-1} \rceil \}$. Take vector $x' \in c_{00}$ such that $x'\geq n+1$ and $\|x'-x\|< \varepsilon$. 
\begin{itemize}
\item[] \emph{Step 1}. There exits a sequence of indieces $\{j_i\}_{i=1}^{\infty}$ and a block sequence $\{y_i\}_{i=1}^{\infty}$  such that $j_1 > t$, $y_1> x'$ and for any $k \in \mathbb{N}$ holds $j_{k+1} > j_k$, $y_{k+1} > \max \{ j_{k},\max \supp \,  y_k\}$ and $\|Ue_{j_k}-y_k\|< \varepsilon$.

We will show this inductively. 

Fix $k=1$. As $(Ue_i)_{i=1}^\infty$ is weakly null we find index $j_1 > t$ and vector $y_1>x'$ so that $ \big\| Ue_{j_1}- y_1\big\|<\varepsilon$. 

Assume that for any $s \leq k$, $s \in \mathbb{N}$ the inductive hypothesis is satisfied. First, applying the inductive hypothesis, determine indieces $t<j_1<\cdots<j_k$ and vectors 
\begin{equation*}
    \begin{split}
         x' &< y_1 \leq \max \{ \, j_1, \max \supp \, y_1 \, \} < y_2
         \leq \cdots \\& \cdots \leq \max \{ \, j_{k-1}, \max \supp \, y_{k-1} \, \} < y_{k}.
    \end{split}
\end{equation*}
As $(Ue_i)_{i=1}^{\infty}$ is weakly null, we can choose $j_{k+1}>j_k$ and a vector $y_{k+1}> \max \{ j_{k},\max \supp \,  y_k\}$ with $ \big\| Ue_{j_{k+1}}- y_{k+1}\big\|<\varepsilon$.

\item[] \emph{Step 2}. 

\emph{Case 1}. Suppose that $\alpha= \beta +1$ for some $\beta < \omega_1$.

Note that every $\mathcal S_{\alpha}$-set whose minimum is $n$ is the union of at most $n$ many $\mathcal S_{\beta}$-sets, so the idea of the proof of this case is to choose the indices $\lceil \theta^{-1} \rceil < j_1<j_2<\cdots<j_m,$ for some $m \in \mathbb{N},$ so that they creates $n$ many maximal and consecutive $\mathcal S_{\beta}$-sets. At the same time, we must ensure that the set \[
\{ \min \supp \, x', \, \min \supp \, y_1, \, \min \supp \, y_2,\dots,\min \supp \, y_m \} 
\]
associated with these indices was also $\mathcal S_{\alpha}$-set, i.e., it was the union of at most (not necessarily maximal) $n+1$ many $\mathcal S_{\beta}$-sets. We proceed as follows.
By the Step $1$ we may choose a~maximal $\mathcal S_{\beta}$-set created from the indices $j_1< j_2 < \dots < j_{m_1},$ for some $m_1 \in \mathbb{N}$ where $j_1 > t$. At the same time, we get vectors $x' < y_1 < y_2 < \dots < y_{m_1}$ so that $\|x'-x\|< \varepsilon$ and $\big\| Ue_{k}- y_k\big\|<\varepsilon$ for $k = {1,\dots, m_1}.$ Next, we apply again the Step $1$ to find the second maximal  $\mathcal S_{\beta}$-set with minimum $j_{m_1+1}$ greater than $j_{m_1}$ and block vectors 
\begin{equation*}
    \begin{split}
         x' &< y_1 \leq \max \{ \, j_1, \max \supp \, y_1 \, \} < y_2
         \leq \cdots \\& \cdots \leq \max \{ \, j_{m_2-1}, \max \supp \, y_{m_2-1} \, \} < y_{m_2},
    \end{split}
\end{equation*}
 for some $m_2\in \mathbb{N}$, so that $\|x'-x\|< \varepsilon$ and $\big\| Ue_{k}- y_k\big\|<\varepsilon$ for $k = {1,\dots, m_2}$. Then, the set 
\[
\{ \min \supp \, y_2, \, \min \supp \, y_3,\dots,\min \supp \, y_{m_1+1} \}\]
is $\mathcal S_{\beta}$-set because the Schreier family (of order $\beta$) is spreading (see Definition \ref{sp}).

Proceeding analogously, we finally arrive at indices 
\[
j_1<j_2 < \cdots <j_{m},
\]
for some $m\in \mathbb{N}$, that form a union of $n$ maximal $\mathcal S_{\beta}$-sets with minima
\[
\big\{ j_1,\, j_{m_1+1}, \dots, j_{m_{n-1}+1}\big\},
\]
so we got the conclusion because we may choose maximal $\mathcal S_{\beta}$-sets

\begin{equation}\label{sets}
        n+1\leq S^1_{\beta} < S^2_{\beta} < \dots < S^{n+1}_{\beta},
\end{equation}

    where
    \begin{itemize}
        \item  
    $
    S^1_{\beta} = \{ \min \supp \, x', \, \min \supp \, y_1 \},
    $
        \item  
    $
    S^2_{\beta} = \{ \min \supp \, y_2,\, \min \supp \, y_3, \ldots, \min \supp \, y_{m_1+1} \},
    $
        \item
    $
    \vdots
    $  
        \item  
    $
    S^{n+1}_{\beta} = \{ \min \supp \, y_{m_{n-1}+2},\, \min \supp \, y_{m_{n-1}+3}, \ldots, \min \supp \, y_{m} \}.
    $
    \end{itemize}
    Note that we profit from the assumption $\alpha \geq 2$ in the fact that $S_{\beta}^1 \in \mathcal S_{\beta}$.

Since the indices $j_1<j_2<\cdots <j_{m}$, chosen as in Step 2 above, form a set that is the union of $n$ maximal and consecutive $\mathcal S_{\beta}$-sets with minimum greater than $n$, by the spreading property of $\mathcal S_{\beta}$ and Lemma \ref{ogl} (as $\{j,j_1,j_2, \ldots j_m \}$ is not an $\mathcal S_{\alpha}$-set), we have
\begin{equation}\label{d}
     \Bigg\|e_j + \sum_{i=1}^{m} e_{j_i}\Bigg\|  = \theta \cdot m.
\end{equation}

Since $U$ is an isometry, we obtain
\begin{equation}\label{dd}
    \Bigg\|e_j + \sum_{i=1}^{m} e_{j_i}\Bigg\| = \Bigg\|\sum_{i=1}^n a^j_i e_i + x + \sum_{i=1}^{m} U e_{j_i}\Bigg\|.
\end{equation}
By \eqref{d}, \eqref{dd}, and the triangle inequality 
\begin{equation}\label{3335}
    \Bigg\|\sum_{i=1}^n a^j_i e_i + x' + \sum_{i=1}^{m} y_i\Bigg\|\leq \theta \cdot m+(m+1)\varepsilon. 
\end{equation}
On the other hand,
\begin{equation}\label{3336}
     \Bigg\|\sum_{i=1}^n a^j_i e_i + x' + \sum_{i=1}^{m} y_i\Bigg\| \geq  \theta \Bigg(\|x'\| + \sum_{i=1}^{m} \|y_i\|\Bigg) > \theta \big(\|x'\| + m - m\varepsilon\big);
\end{equation}
the former inequality follows from \eqref{w1} as we may choose sets as in \eqref{sets},
whereas the latter one holds because $\big\|Ue_i\big\|=1$ ($i\in \mathbb{N}$). Thus, by \eqref{3335} and \eqref{3336}, we have
\[
\|x\|\leq \|x'-x\|+\|x'\| < \varepsilon + \big(\theta^{-1}(m+1)+m\big)\varepsilon,
\] 
so $\|x\|=0$. Consequently, $Ue_j= a^j_1 e_1 + a^j_2 e_2+ \dots + a^j_n e_n.$ \smallskip

\item[] \emph{Case 2}: Suppose that $\alpha$ is a limit ordinal.

We proceed as in Case $1$ for $\alpha= \alpha_n,$ where $(\alpha_i)_{i=1}^{\infty}$ is a fixed strictly increasing sequence of successor ordinals converging to $\alpha$ with  $\mathcal S_{\beta_i} \subset \mathcal S_{\beta_k}$ for $i \leq k$, where $\alpha_n := \beta_{n} +1$ for each $n \in \mathbb{N}$, choosing suitable sequences $(j_k)_{k=1}^m$ and $(y_k)_{k=1}^m$.
        Indeed, an $\mathcal S_{\alpha}$-set whose minimum is $n$ must be an $\mathcal S_{\alpha_n}$-set and $\mathcal S_{\beta_n}$-sets $n+1\leq S^1_{\beta_n} < S^2_{\beta_n} < \dots < S^{n+1}_{\beta_n}$, where
    \begin{itemize}
        \item  
    $
    S^1_{\beta_n} = \{ \min \supp \, x', \, \min \supp \, y_1 \},
    $
        \item  
    $
    S^2_{\beta_n} = \{ \min \supp \, y_2,\, \min \supp \, y_3, \ldots, \min \supp \, y_{m_1+1} \},
    $
        \item
    $
    \vdots
    $  
        \item  
    $
    S^{n+1}_{\beta_n} = \{ \min \supp \, y_{m_{n-1}+2},\, \min \supp \, y_{m_{n-1}+3}, \ldots, \min \supp \, y_{m} \},
    $
    \end{itemize}
    give rise to an $\mathcal S_{\alpha}$-set (even an $\mathcal S_{\alpha_n}$-set). Moreover, the set $A:= \{j, j_1, j_2, \ldots, j_m\}$ is not an $\mathcal S_{\alpha}$-set, as we ensured that $\mathcal S_{\beta_i} \subset \mathcal S_{\beta_k}$ for $i \leq k$. Indeed, suppose $A \in \mathcal S_{\alpha}$, then $A \in \mathcal S_{\alpha_j}$ i.e. $A$ is the union of at most $j$-many successive $\mathcal S_{\beta_j}$-sets, i.e. $\mathcal S_{\beta_n}$-sets by the assumption on $(\beta_n)_{n}$. This contradicts the choice of $\{j_1,j_2, \ldots, j_m\}$ as the union of $n$-many maximal $S_{\beta_n}$-sets by the spreading property of $S_{\beta_n}$.
\end{itemize}
\noindent\emph{Claim 2}. $Ue_n = \pm e_n$ for $n\in \mathbb{N}.$   
\begin{itemize}
    \item[]
Set $n=1$. Then by Claim 1 we have $Ue_{1}= a^{1}_1 e_1.$ Since  $\big\|Ue_1\big\|=1$, so $Ue_{1}=\pm e_{1}.$

 We now proceed to the inductive step. Fix $n\in \mathbb{N}$ and assume that for $k\in \mathbb{N}$ with $k < n$ one has $Ue_k=\pm e_k$. Then by Claim 1 we have $Ue_{n}= a^{n}_1 e_1 + a^{n}_2 e_2 +  \dots + a^{n}_{{n}}e_{n}.$ 
By the very definition of the norm and the inductive hypothesis we have
\begin{equation*}
    1=\max \{1,\, 2 \cdot \theta\} = \|e_k\pm e_{n}\|= \big\|e_{k}\pm Ue_{n}\big\| \geq \big\|e_{k}\pm Ue_{n}\big\|_{\infty} \geq \big|1\pm a^{n}_k\big|
\end{equation*}
for $k\in \{1,2,\dots, n-1\}.$ Hence $a^{n}_{k}=0$ for all such $k.$  Since $\big\|Ue_i\big\|=1$ ($i\in \mathbb{N}$), so $Ue_{n}=\pm e_{n}.$ 
\end{itemize}
\end{proof}

\bibliography{bibliography.bib}

\begin{thebibliography}{1}

\bibitem{antunes2022surjective}
L.~Antunes and K.~Beanland.
\newblock Surjective isometries on {Banach} sequence spaces: A survey.
\newblock {\em Concrete Operators}, 9(1):19--40, 2022.

\bibitem{BrechFerencziTcaciuc2020}
C.~Brech, V.~Ferenczi, and A.~Tcaciuc.
\newblock Isometries of combinatorial {Banach} spaces.
\newblock {\em Proc. Am. Math. Soc.}, 148(11):4845--4854, 2020.

\bibitem{causey2017concerning}
Ryan~M Causey.
\newblock Concerning the szlenk index.
\newblock {\em Studia Mathematica}, 236:201--244, 2017.

\bibitem{FigielJohnson1974}
T.~Figiel and W.~B. Johnson.
\newblock A uniformly convex {Banach} space which contains no {{\(\ell_p\)}}.
\newblock {\em Compos. Math.}, 29:179--190, 1974.

\bibitem{todorcevic2005ramsey}
S.~Todor\v{c}evi\'c and S.~Argyros.
\newblock {\em Ramsey Methods in Analysis}.
\newblock Advanced Courses in Mathematics - CRM Barcelona. Birkh\"auser, 2005.

\bibitem{tsirelson1974impossible}
B.~Tsirelson.
\newblock It is impossible to embed $\ell_p$ or $c_0$ into an arbitrary
  {Banach} space ({Russian}).
\newblock {\em {Funkts. Anal. i Prilozhen} English translation: {Funct. Anal.
  Appl}}, 8:138--141, 1974.

\end{thebibliography}
\bibliographystyle{plain}

\end{document}